\documentclass[reqno]{amsart}
\usepackage[sorted-cites]{amsrefs}

\usepackage{amsmath,amssymb,amsthm, epsfig}
\usepackage{amsfonts}
\usepackage{amsmath}
\usepackage{amssymb}
\usepackage[mathscr]{euscript}

\newtheorem*{acknowledgement}{Acknowledgement}
\newtheorem{corollary}{Corollary}

\newtheorem{lemma}{Lemma}

\newtheorem{remark}{Remark}
\newtheorem{theorem}{Theorem}
\newtheorem{example}{Example}
\numberwithin{equation}{section}

\renewcommand{\div}{{\rm div}}

\begin{document}

\title[quasi-Einstein manifolds with boundary]{Geometry of compact quasi-Einstein\\ manifolds with boundary}
\author[R. Di\'ogenes]{Rafael Di\'ogenes}
\author[T. Gadelha]{Tiago Gadelha}
\author[E. Ribeiro Jr]{Ernani Ribeiro Jr}

\address[R. Di\'ogenes]{UNILAB, Instituto de Ci\^encias Exatas e da Natureza, Rua Jos\'e Franco de Oliveira, 62790-970, Reden\c{c}\~ao - CE, Brazil.}\email{rafaeldiogenes@unilab.edu.br}

\address[T. Gadelha]{Instituto Federal do Cear\'a - IFCE, Campus Maracana\'u, Av. Parque Central, 61939-140, Maracana\'u - CE, Brazil.}\email{tiago.gadelha@ifce.edu.br}

\address[E. Ribeiro Jr]{Universidade Federal do Cear\'a - UFC, Departamento  de Matem\'atica, Campus do Pici, Av. Humberto Monte, 60455-760, Fortaleza - CE, Brazil.}
\email[Corresponding author]{ernani@mat.ufc.br}

\thanks{T. Gadelha was partially supported by FUNCAP/Brazil}

\thanks{E. Ribeiro was partially supported by CNPq/Brazil [Grants: 305410/2018-0 \& 160002/2019-2] and CAPES/Brazil - Finance Code 001}

\thanks{Corresponding Author: E. Ribeiro (ernani@mat.ufc.br)}

\begin{abstract}
 In this article, we study the geometry of compact quasi-Einstein manifolds with boundary. We establish sharp boundary estimates for compact quasi-Einstein manifolds with boundary that improve some previous results. Moreover, we obtain a characterization theorem for such manifolds in terms of the surface gravity of the boundary components, which leads to a new sharp geometric inequality. In addition, we prove a boundary estimate for compact quasi-Einstein manifolds with (possibly disconnected) boundary in terms of the Brown-York mass.
 \end{abstract}

\date{\today}



\maketitle

\section{Introduction}
\label{intro}

According to the approach used by Case, Shu and Wei \cite{CaseShuWey}, and He, Petersen and Wylie \cite{He-Petersen-Wylie2012,Petersen-Chenxu}, a complete Riemannian manifold $(M^n,\,g),$ $n\geq 2,$ possibly with boundary $\partial M,$ will be called $m$-{\it quasi-Einstein manifold}, or simply {\it quasi-Einstein manifold}, if there exists a smooth potential function $u$ on $M^n$ obeying the following system
\begin{equation}
\label{eqdef}
\left\{%
\begin{array}{lll}
    \displaystyle \nabla^{2}u = \dfrac{u}{m}(Ric-\lambda g) & \hbox{in $M,$} \\
    \displaystyle u>0 & \hbox{on $int(M),$} \\
        \displaystyle u=0 & \hbox{on $\partial M,$} \\
    \end{array}%
\right.
\end{equation} for some constants $\lambda$ and $0<m<\infty.$ When $m=1$ we make the additional condition $\Delta u=-\lambda u$ in order to recover the static equation 
\begin{equation}
-(\Delta u)g+\nabla^{2}u -uRic =0.
\end{equation} Here, $\nabla^{2} u$ stands for the Hessian of $u$ and $Ric$ is the Ricci tensor of $g.$ We say that a quasi-Einstein manifold is {\it trivial} if its potential function $u$ is constant, otherwise, we say that it is {\it nontrivial}.

Our motivation for studying such manifolds is that an $m$-quasi-Einstein manifold corresponds to a base of a warped product Einstein metric; for more details see \cite[pg. 265]{Besse} (cf. \cite{He-Petersen-Wylie2012,Petersen-Chenxu}). Hence, it is directly related to the existence of Einstein metrics on a given manifold. Another motivation to investigate quasi-Einstein manifolds comes from the study of diffusion operators by Bakry and \'Emery \cite{bakry}, which is closely tied to the theory of smooth metric measure spaces. Moreover, $1$-quasi-Einstein manifolds are more commonly called {\it static spaces} and this terminology relies on the physical nature of the problem associated to the va\-cuum Einstein field equations; for more details, see \cite[Remark 2.3]{CaseShuWey}. Explicit examples of compact and noncompact quasi-Einstein manifolds can be found in, e.g., \cite{Besse,CaseP,CaseT,CaseShuWey,He-Petersen-Wylie2012,LuePage,Ernani_Keti,Rimoldi,Wang}.

He, Petersen and Wylie \cite{He-Petersen-Wylie2012,Petersen-Chenxu,HPW3} developed a successful study on quasi-Einstein manifolds with nonempty boundary. Among other things, they obtained a classification result for quasi-Einstein manifolds which are also Einstein.  They also studied quasi-Einstein manifolds with constant scalar curvature. The locally conformally flat case was discussed in \cite{CMMR,He-Petersen-Wylie2012}. In this article, we focus on nontrivial compact $m$-quasi-Einstein manifolds with boundary. Hence, by Theorem 4.1 of \cite{He-Petersen-Wylie2012}, they necessarily have $\lambda>0.$

At this point, it is useful to highlight some examples. Let us start with the standard hemisphere $\Bbb{S}^n_+.$

\begin{example}[\cite{Petersen-Chenxu}]
\label{example1}
Let  $\Bbb{S}^n_+$ be a standard hemisphere with metric $g=dr^2+\sin^2r g_{\Bbb{S}^{n-1}}$ and potential function $u(r)=\cos r,$ where $r$ is a height function with $r\leq\frac{\pi}{2}.$ Thus, $\Bbb{S}^n_+$ is a compact $m$-quasi-Einstein manifold with bo\-undary $\Bbb{S}^{n-1}.$ 
\end{example}

The next example has disconnected boundary.

\begin{example}[\cite{DG 2019}]
\label{example2}
Let $M=\Big[0,\sqrt{m}\pi\Big]\times\Bbb{S}^{n-1}$ be a Riemannian product with metric $g=dt^2+(n-2)g_{\Bbb{S}^{n-1}}$ and potential function $u(t,x)=\sin\left(\dfrac{1}{\sqrt{m}}t\right).$ Thus, $M$ is a compact $m$-quasi-Einstein manifold with disconnected boundary. Notice that the boundary is the union of two copies of $\Bbb{S}^{n-1}.$ 
\end{example}

We remark that if a nontrivial compact $m$-quasi-Einstein manifold with nonempty boundary has constant scalar curvature $R,$ then $R< n\lambda$ (see \cite[Corollary 4.3]{He-Petersen-Wylie2012}). Specifically, Example \ref{example1} has $R-n\lambda =-mn$ and Example \ref{example2} satisfies $R-n\lambda=-1.$

Boundary estimates are classical objects of study in geometry and physics. Besides being interesting on their own, such estimates are useful in proving new classification results and discarding some possible new examples  of special metrics on a given manifold. Among the contributions that motivated this work, we primarily mention the classical isoperimetric inequality and a result due to Shen \cite{Shen} and Boucher, Gibbons and Horowitz \cite{BGH} that asserts that the boundary $\partial M$ of a compact three-dimensional oriented static space (i.e., a $1$-quasi-Einstein ma\-ni\-fold) with connected boundary and scalar curvature $6$ must be a $2$-sphere whose area satisfies the inequality $|\partial M |\leq 4\pi,$ with equality if and only if $M^3$ is equivalent to the standard hemisphere. In the same spirit, boundary estimates for $V$-static metrics and static spaces were es\-ta\-blished in, e.g., \cite{Lucas,BalRi2,BS,Batista-Diogenes-Ranieri-Ribeiro JR,CGP,CEM,miaotam,HMR,Yuan}. In the recent work \cite[Theorem 1]{DG 2019}, Di\'ogenes and Gadelha proved an analogous boundary estimate for compact $m$-quasi-Einstein manifolds $M^n$ with connected boundary $\partial M$ by assuming the following conditions:
\begin{enumerate}
\item[(1)] $Ric^{\partial M}\geq \dfrac{R^{\partial M}}{n-1}g_{_{\partial M}},$

\item[(2)] $\inf_{\partial M} R^{\partial M}>0,$ and

\item[(3)] $R\leq n\lambda,$ where $R$ stands for the scalar curvature of $M^n.$  

\end{enumerate}  In \cite[Theorem 2]{FS}, Freitas and Santos showed that the same boun\-dary estimate holds by replacing the condition $(3)$ by the assumption $\mathcal{L}_{\nabla u}R\geq 0,$ where $\mathcal{L}_{\nabla u}R$ denotes the Lie derivative of the scalar curvature $R$ in the direction of $\nabla u.$ Besides, since $Ric^{\partial M}- \dfrac{R^{\partial M}}{n-1}g_{_{\partial M}}$ is trace-free, then the condition $(1)$ is equivalent to the assumption that the boundary is Einstein with the induced metric.

In this article, we shall provide new sharp boundary estimates for compact quasi-Einstein manifolds with boundary. To start with, we establish a rigidity result for the boundary by removing the aforementioned conditions $(2)$ and $(3)$ considered in \cite[Theorem 1]{DG 2019} and \cite[Theorem 2]{FS}. More precisely, we have the following result.

\begin{theorem}
\label{th1semhyp}
Let $\big(M^{n},\,g,\,u,\,\lambda \big),$ $n\geq 3,$ be a nontrivial compact, oriented $m$-quasi-Einstein manifold with connected boundary and $m>1.$ Then the following assertions hold:

\begin{enumerate}
\item[(1)]  If $(\partial M,\,g_{\partial M})$ is Einstein, then
\begin{equation}
\label{eq34t}
|\partial M | \leq \left(\frac{m+n-1}{\lambda}\right)^{\frac{n-1}{2}}\,\omega_{n-1},
\end{equation} where $\omega_{n-1}$ is the volume of the standard unitary sphere $\Bbb{S}^{n-1}.$ Moreover, equality holds if and only if $\partial M$ is isometric to the round sphere $\Bbb{S}^{n-1}\Big(\sqrt{\frac{m+n-1}{\lambda}}\Big).$

\item[(2)] If $n=3,$ then 
\begin{equation}
\label{est3}
|\partial M |  \leq \frac{(m+2)}{\lambda}4\pi.
\end{equation} Moreover, equality holds if and only if $\partial M$ is isometric to the round sphere $\Bbb{S}^{2}\Big(\sqrt{\frac{m+2}{\lambda}}\Big).$

\end{enumerate}
\end{theorem}

One new feature is that no scalar curvature condition is assumed in Theorem \ref{th1semhyp}. Notice further that (\ref{est3}) implies that the area of the boundary of a three-dimensional hemisphere is the maxi\-mum possible among all three-dimensional compact $m$-quasi-Einstein manifolds with connected boundary.  It would be interesting to see if
our estimates (\ref{eq34t}) and (\ref{est3}) imply the rigidity of the interior as well.

In order to introduce the next result, we need to fix a couple of notations. Given a compact Riemannian manifold $\big(M^{n},\,g\big)$ with (possibly disconnected) boundary $\partial M$ and a smooth function $u$ on $M^n,$ we consider $u_{\max}=\max_{M}u$ and for a given connected component $\partial M_{i}\subset \partial M$ of the boundary we let $$\kappa(\partial M_{i})=\frac{|\nabla u|_{| \partial M_{i}}}{u_{\max}}$$ be the so called {\it surface gravity} of $\partial M_{i}.$ In the Newtonian case, the surface gravity of a rotationally symmetric massive body (e.g., a planet of the solar system) can be physically interpreted as the intensity of the gravitational field due to the
body. For further insights about the physical meaning of this concept, see \cite[Section 12.5]{Wald} and \cite{BM1,BM2}. We also point out that the surface gravity is locally constant on $m$-quasi-Einstein manifolds.

In the sequel, we obtain a characterization of compact quasi-Einstein manifolds with (possibly disconnected) boundary in terms of the surface gravity of the connected components $\partial M_{i}$ of $\partial M.$

\begin{theorem}
\label{thmK}
Let $\big(M^{n},\,g,\,u,\,\lambda \big),$ $n\geq 3,$ be a nontrivial compact $m$-quasi-Einstein manifold with (possibly disconnected) boundary $\partial M$ and constant scalar curvature. Then we have $$\max_{i}\kappa(\partial M_{i})\geq \sqrt{\frac{n\lambda-R}{mn}}.$$
Moreover, if equality holds, then $M^n$ is homothetic to the standard hemisphere $\Bbb{S}^n_+$ given by Example \ref{example1}.
\end{theorem}

Observe that Theorem \ref{thmK} implies that the standard hemisphere $\Bbb{S}^n_+$ has the least possible surface gravity among all compact $m$-quasi-Einstein manifolds with connected boundary $\partial M$ and constant scalar curvature. Moreover, it can be seen as a gradient estimate for the potential function $u.$ Furthermore, since a $1$-quasi-Einstein manifold is precisely a static space, which necessarily has constant scalar curvature $R=(n-1)\lambda,$ Theorem \ref{thmK} extends \cite[Theorem 2.1]{BM1} to the $m>1$ case.

As a consequence of the proof of Theorem \ref{thmK}, we obtain the following sharp geometric inequality involving the area of the boundary and volume of a compact quasi-Einstein manifold with connected boundary, which can also be interpreted as an obstruction result.

\begin{corollary}
\label{TcorB}
Let $\big(M^{n},\,g,\,u,\,\lambda \big),$ $n\geq 3,$ be a nontrivial compact, oriented $m$-quasi-Einstein manifold with connected boundary, constant scalar curvature and $m>1.$ Then we have:

\begin{equation}
\label{eq5t}
|\partial M|\leq n\sqrt{\frac{\lambda}{m+n-1}}\, Vol(M).
\end{equation} Mo\-reover, equality holds if and only if $M^n$ is homothetic to the standard hemisphere $\Bbb{S}^n_+$ given by Example \ref{example1}. 
\end{corollary}

Before presenting our next result, let us recall the concept of {\it Brown-York mass} given by $$\mathfrak{m}_{BY}(\partial M,\,g)=\int_{\partial M}\left(H_{0}-H_{g}\right) dS_{g},$$ where $H_{g}$ is the mean curvature of $\partial M$ with respect to the metric $g,$ $dS_{g}$ is the volume element induced on $\partial M$ by $g$ and $H_{0}$ is the mean curvature of $\partial M$ with respect to the outward unit normal when embedded in $\Bbb{R}^n.$

In our next result, we will obtain a boundary estimate for compact quasi-Einstein manifolds with (possibly disconnected) boundary in terms of the Brown-York mass. In particular, no scalar curvature condition will be assumed. To do so, we will make use of the positive mass theorem for Brown–York mass due to Shi–Tam \cite{ShiTam}, which is equivalent to the (higher dimensional) positive mass theorem for ADM mass by Schoen and Yau \cite{SY1,SY2,SY3} and Lohkamp \cite{Lohkamp}. More precisely, we have the following result.
	
	\begin{theorem}
		\label{thmBY}
		Let $\big(M^{n},\,g,\,u,\,\lambda \big),$ $n\geq 3,$ be a nontrivial compact $m$-quasi-Einstein manifold with (possibly disconnected) bo\-un\-dary $\partial M=\cup_{i=1}^{k}\partial M_{i}$ and $m > 1$. Suppose that each boundary component $(\partial M_{i},\,g)$ can be isometrically embedded in $\Bbb{R}^n$ as a convex hypersurface. Then we have
		\begin{equation}
		|\partial M_{i}|\leq c\,\mathfrak{m}_{BY}(\partial M_{i},\,g),
		\end{equation} where $c$ is a positive constant. Moreover, equality holds for some component $\partial M_{i}$ if and only if $M^n$ is homothetic to the standard hemisphere $\Bbb{S}^n_+$ given by Example \ref{example1}. 
	\end{theorem}

\begin{remark} The isometrical embedding assumption in Theorem \ref{thmBY} was necessary to apply the positve mass theorem. However, according to the solution of the Weyl problem, the isometrical embedding assumption can be replaced by assumptions on sectional curvatures, as for instance, the Gaussian curvature of $\partial M$ is positive when $n = 3$ (see \cite{Yuan,YuanW}).
\end{remark}

\vspace{0.40cm}

The organization of the paper is as follows. In Section \ref{preliminaries}, we review a few preliminary facts about compact quasi-Einstein manifolds with boundary. In Section \ref{secProofsA}, we present the proof of Theorem \ref{th1semhyp}. Theorem \ref{thmK} and Corollary \ref{TcorB} are proved in Section \ref{secProofsB0}. In Section \ref{secProofBY}, we present the proof of Theorem \ref{thmBY}.

\section{Background}
\label{preliminaries}

In this section, we will review some basic facts and key lemmas that will be useful for the establishment of the main results.

We start by remembering that the fundamental equation of an $m$-quasi-Einstein manifold $(M^{n},\,g,\,u,\,\lambda),$ possibly with boundary, is given by 
\begin{equation}\label{fundamental equation}
\nabla ^{2}u = \dfrac{u}{m}(Ric-\lambda g),
\end{equation} where $u>0$ in the interior of $M^n$ and $u=0$ on $\partial M.$ In particular, taking the trace of (\ref{fundamental equation}) we arrive at
\begin{equation}\label{laplaciano}
\Delta u = \frac{u}{m}(R-\lambda n).
\end{equation} Plugging this into (\ref{fundamental equation}), one sees that
\begin{equation}\label{traceless Ricci}
u\, \mathring{Ric} = m \mathring{\nabla^2}\, u,
\end{equation} where $\mathring{T}=T-\dfrac{{\rm tr}T}{n}g$ stands for the traceless part of $T.$

Since $u>0$ in the interior of $M^n$ and $u=0$ on the boundary $\partial M,$ one obtains that $N=-\frac{\nabla u}{|\nabla u|}$ is the outward unit normal vector. Besides, it follows from Propositions 2.2 and 2.3 in \cite{He-Petersen-Wylie2012} that $|\nabla u|\neq0$ is constant along $\partial M.$

From now on, we set an orthonormal frame given by $$\left\{e_1,\ldots,e_{n-1},e_n=-\frac{\nabla u}{|\nabla u|}\right\}.$$ Hence, the second fundamental form at $\partial M$ satisfies
\begin{eqnarray}\label{fundamental second}
h_{ij}=\langle \nabla _{e_{i}}N, e_{j}\rangle= -\frac{1}{|\nabla u|}\nabla _{i}\nabla _{j} u=0,
\end{eqnarray} for any $1\leq i,j\leq n-1.$ Thereby, $\partial M$ is totally geodesic. By Gauss equation $$R^{\partial M}_{ijkl}=R_{ijkl}-h_{il}h_{jk}+h_{ik}h_{jl},$$ we then infer
\begin{eqnarray}
 \label{GaussEq}
R^{\partial M}_{ijkl}&=&R_{ijkl},
\end{eqnarray}

\begin{eqnarray}
\label{RicciBordoM}
R^{\partial M}_{ik}&=&R_{ik}-R_{inkn}
\end{eqnarray}
and
\begin{eqnarray}
\label{EscalarBordoM}
R^{\partial M}&=&R-2R_{nn}.
\end{eqnarray}  While the twice-contracted second Bianchi identity ($2\div Ric=\nabla R$) yields
\begin{equation}\label{divtracelessRic}
\div\mathring{Ric}=\frac{n-2}{2n}\nabla R,
\end{equation} where $\mathring{Ric}$ stands for the traceless Ricci tensor.

Next, it is important to recall that He, Petersen and Wylie \cite[Remark 5.1]{He-Petersen-Wylie2012} proved that the scalar curvature of a nontrivial compact quasi-Einstein manifold with boundary must satisfy

\begin{equation}
\label{eq3e4}
R\geq\frac{n(n-1)}{m+n-1}\lambda.
\end{equation} See also \cite[Proposition 3.6]{CaseShuWey}.

\section{The Proof of Theorem \ref{th1semhyp}}
\label{secProofsA}

\vspace{0.3cm}

In this section, we shall present the proof of Theorem \ref{th1semhyp}. To this end, we need to establish the following ingredient.

\begin{lemma}
\label{propRboundary}
Let $(M^n,\,g,\,u,\lambda)$ be a compact $m$-quasi-Einstein manifold with bo\-undary $\partial M$ and $m\neq1.$ Then we have:
\begin{enumerate}
\item $Ric(\nabla u)=-\frac{R-(n-1)\lambda}{m-1}\nabla u$ on the boundary $\partial M;$ 
\item $(m+1)R=(m-1)R^{\partial M}+2(n-1)\lambda$ on the boundary $\partial M.$ 
\end{enumerate}
\end{lemma}

\begin{proof} By the twice-contracted second Bianchi identity ($2 \div Ric=\nabla R$) and (\ref{fundamental equation}), one obtains that 

\begin{eqnarray}
\label{eqrft4}
\frac{1}{2}u\nabla R&=&u\div Ric=\div (uRic)-Ric(\nabla u)\nonumber\\
 &=&m\div (\nabla^2u)+\lambda\nabla u-Ric(\nabla u).
\end{eqnarray} 

On the other hand, recall that, in general, for any smooth function $f$ on a Riemannian manifold $(M^n,\,g),$ we have
\begin{equation*}
\div\nabla^2f=Ric(\nabla f)+\nabla\Delta f.
\end{equation*} Substituting this data jointly with (\ref{laplaciano}) into (\ref{eqrft4}), we infer

\begin{eqnarray*}
\frac{1}{2}u\nabla R&=&mRic(\nabla u)+m\nabla\Delta u+\lambda\nabla u-Ric(\nabla u)\\
 &=&(m-1)Ric(\nabla u)+u\nabla R+R\nabla u-n\lambda\nabla u+\lambda\nabla u,
\end{eqnarray*}  so that

\begin{equation}
\label{9kip}
-\frac{1}{2}u\nabla R=(m-1)Ric(\nabla u)+\left(R-(n-1)\lambda\right)\nabla u.
\end{equation} Since $u$ vanishes on the boundary, one concludes that
\begin{equation}\label{propRboundaryeq1}
(m-1)Ric(\nabla u)=-\left(R-(n-1)\lambda\right)\nabla u
\end{equation} on the boundary $\partial M.$ This proves the first assertion.

Next, we use (\ref{propRboundaryeq1}) to infer $$Ric(\nabla u,\nabla u)=-\frac{R-(n-1)\lambda}{m-1}|\nabla u|^2,$$ so that
\begin{equation*}
R_{nn}=-\frac{R-(n-1)\lambda}{m-1}.
\end{equation*} This combined with (\ref{EscalarBordoM}) gives
\begin{equation*}
\frac{R-R^{\partial M}}{2}=-\frac{R-(n-1)\lambda}{m-1}
\end{equation*} and hence, 
\begin{equation*}
(m+1)R=(m-1)R^{\partial M}+2(n-1)\lambda
\end{equation*} on the boundary $\partial M,$ which proves the desired result. 
\end{proof}

Now, we are going to conclude the proof of Theorem \ref{th1semhyp}.

\subsection{Conclusion of the proof of Theorem \ref{th1semhyp}}
\begin{proof}
To begin with, plugging (\ref{eq3e4}) into the second item of Lemma \ref{propRboundary}, we obtain

\begin{eqnarray*}
(m-1)R^{\partial M}&=&(m+1)R-2(n-1)\lambda\\
 &\geq&\frac{(m+1)n(n-1)}{m+n-1}\lambda-2(n-1)\lambda\\
 &=&\frac{(m-1)(n-1)(n-2)}{m+n-1}\lambda.
\end{eqnarray*} In particular, taking into account that $m>1,$ we deduce
\begin{equation}\label{th1semhypeq1}
R^{\partial M}\geq\frac{(n-1)(n-2)}{m+n-1}\lambda.
\end{equation} Since $M^n$ is compact, we already know that $\lambda>0$ (see \cite[Theorem 4.1]{He-Petersen-Wylie2012}) and therefore, (\ref{th1semhypeq1}) yields

\begin{equation}
\label{plk7}
\inf_{\partial M}R^{\partial M}>0.
\end{equation}

We now claim that

\begin{eqnarray}
\label{eqhjq}
	(n-1)(n-2)\big (\omega_{n-1}\big )^{\frac{2}{n-1}}\geq R^{\partial M} |\partial M|^{\frac{2}{n-1}}. 
\end{eqnarray} Indeed, since $(\partial M,\,g_{\partial M})$ is Einstein,  it follows from (\ref{plk7}) that there exists a constant $\theta>0$ such that $Ric^{\partial M}= (n-2)\theta,$ where $\theta=\frac{R^{\partial M}}{(n-1)(n-2)}.$ At the same time, by Bonnet-Myers theorem, we have $diam_{g_{_{\partial M}}}(\partial M) \leq \frac{\pi}{\sqrt{\theta}}.$ Moreover, by Bishop-Gromov comparison theorem one obtains that

\begin{equation}
\label{mjk1}
|\partial M|\leq |B_{\frac{\pi}{\sqrt{\theta}}}(p)|\leq |\Bbb{S}^{n-1}|_{g_{_\theta}}=\theta^{-\frac{(n-1)}{2}}\omega_{n-1},
\end{equation} where $p\in \partial M,$ $g_{_\theta}=\theta^{-1}g_{can}$ and $\omega_{n-1}$ is the volume of the standard unitary sphere $\Bbb{S}^{n-1}.$ Hence,  (\ref{eqhjq}) is in fact true.

Next, substituting (\ref{th1semhypeq1}) into (\ref{eqhjq}) we arrive at 
\begin{eqnarray*}
\lambda\frac{(n-1)(n-2)}{m+n-1}\, |\partial M |^{\frac{2}{n-1}}\leq (n-1)(n-2)\big (\omega_{n-1}\big )^{\frac{2}{n-1}},
\end{eqnarray*} so that
\begin{equation}
\label{lkmnb1}
|\partial M|\leq\omega_{n-1}\left(\frac{m+n-1}{\lambda}\right)^{\frac{n-1}{2}},
\end{equation} which proves the desired estimate.

Proceeding, if equality holds in (\ref{lkmnb1}), then (\ref{mjk1}) also becomes equality. So, it follows from the equality case of the Bishop-Gromov theorem that $\partial M$ is isometric to the sphere $\Bbb{S}^{n-1}\Big(\sqrt{\frac{m+n-1}{\lambda}}\Big).$ The reciprocal statement is obviously true. This therefore finishes the proof of the first assertion.

We now prove the second one. Choosing $n=3$ in Eq. (\ref{th1semhypeq1}), we obtain
\begin{eqnarray}
\label{pyu1}
R^{\partial M}\geq\frac{2}{m+2}\lambda.
\end{eqnarray} From this,  one sees that

\begin{eqnarray*}
\int_{\partial M}R^{\partial M}dS\geq\frac{2\lambda}{m+2}|\partial M|.
\end{eqnarray*} Consequently, it follows from the Gauss-Bonnet theorem that $\partial M$ is a $2$-sphere and
\begin{eqnarray*}
|\partial M|\leq\frac{4(m+2)\pi}{\lambda}.
\end{eqnarray*} In addition, if equality holds, we deduce from (\ref{pyu1}) that $\partial M$ has constant Gaussian curvature and therefore, the result follows from Corollary 1 in \cite{Chern}. The converse assertion is straightforward. So, the proof of the theorem is finished.
\end{proof}

\section{The Proofs of Theorem \ref{thmK} and  Corollary \ref{TcorB}}
\label{secProofsB0}

The proof of Theorem \ref{thmK} is divided into two parts and takes a similar strategy to the one by Borghini-Mazzieri \cite{BM1}. We first have to obtain a Robinson-Shen type identity for quasi-Einstein manifolds with constant scalar curvature in a fashion designed to perform an argument based on the Maximum Principle. Accordingly, throughout this section, we assume that $(M^n,\,g,\,u,\,\lambda)$ is a quasi-Einstein manifold with constant scalar curvature.

To begin with, we claim that
 
 \begin{equation}
 \label{epkl1} \nabla \Psi=2u\mathring{Ric}(\nabla u),
 \end{equation} where $$\Psi:=m|\nabla u|^{2}+\frac{(n\lambda -R)}{n}u^{2}.$$ To prove this, observe first that by (\ref{traceless Ricci}),

\begin{eqnarray*}
u\mathring{Ric}(\nabla u)&=& m\mathring{\nabla}^2 u(\nabla u)=m\nabla^{2}u(\nabla u)-\frac{m\Delta u}{n}(\nabla u)\nonumber\\&=&\frac{m}{2}\nabla|\nabla u|^{2}-\frac{(R-n\lambda)}{n}u\nabla u,
\end{eqnarray*} where we have used (\ref{laplaciano}). In view of this, it then follows that

\begin{eqnarray*}
u\mathring{Ric}(\nabla u)= \frac{1}{2}\nabla \left(m|\nabla u|^{2}+\frac{n\lambda -R}{n}u^{2}\right)= \frac{1}{2}\nabla\Psi,
\end{eqnarray*} as asserted.

Proceeding, a direct computation shows

\begin{eqnarray*}
\Delta \Psi &=& \div \left(2u\mathring{Ric}(\nabla u)\right)\nonumber\\&=& 2u \div \mathring{Ric}(\nabla u)+2u\langle \mathring{Ric},\nabla^{2}u\rangle +2\mathring{Ric}(\nabla u,\nabla u).
\end{eqnarray*} Since $M^n$ has constant scalar curvature, we use the twice-contracted second Bianchi identity, (\ref{epkl1}) and (\ref{traceless Ricci}) to infer

\begin{equation}
\label{eq76u}
\Delta \Psi -\frac{1}{u}\langle \nabla \Psi,\nabla u\rangle=\frac{2}{m}u^{2}|\mathring{Ric}|^2.
\end{equation} This is precisely the desired Robinson-Shen type identity.

\vspace{0.3cm}
Now, we are ready to conclude the proof of Theorem \ref{thmK}.
\vspace{0.3cm}

\subsection{Conclusion of the proof of Theorem \ref{thmK}}

\begin{proof} Initially, we invoke (\ref{eq76u}) to infer

\begin{equation}
\Delta \Psi -\frac{1}{u}\langle \nabla \Psi,\nabla u\rangle\geq 0.
\end{equation} Observe, however, that the coefficient $1/u$ blows up at the boundary. Therefore, before using the Maximum Principle, we need to overcome this technical difficulty by performing the following computation. First, it follows from \cite[Proposition 2.4]{He-Petersen-Wylie2012} that $g$ and $u$ are real analytic in harmonic coordinates. Thereby, we can choose a positive number $\alpha$ such that, for $0<\varepsilon \leq \alpha,$ each  level set $\{x\in M;\,u(x)=\varepsilon\}$ is regular. Considering $M_{\varepsilon}=\{u\geq \varepsilon\},$ one sees that the coefficient $1/u$ is bounded above by $1/\varepsilon$ in $M_{\varepsilon}.$ Hence, applying the Maximum Principle, one concludes that

$$\max_{M_{\varepsilon}}\Psi\leq\max_{\partial M_{\varepsilon}}\Psi.$$ Moreover, we have $$\max_{\partial M_{\alpha}}\Psi\leq\max_{\partial M_{\varepsilon}}\Psi,$$ for every $0<\varepsilon \leq \alpha,$ and

$$\displaystyle\lim_{\varepsilon\rightarrow0^{+}}\max_{\partial M_{\varepsilon}}\Psi=m |\nabla u|^{2}_{|\partial M},$$ where $|\nabla u|_{| \partial M}$ denotes $|\nabla u|$ restricted to the boundary $\partial M.$

From now one, suppose that

\begin{equation}
\label{eqght}
\frac{|\nabla u|^{2}_{|\partial M}}{u_{\max}^{2}}\leq \frac{(n\lambda-R)}{mn}\,\,\,\,\,\,\,\,\,\,\,\hbox{on}\,\,\,\,\,\partial M.
\end{equation} In view of this, one obtains that

 $$\max_{\partial M_{\alpha}}\Psi\leq m|\nabla u|^{2}_{|\partial M} \leq \frac{(n\lambda-R)}{n}u_{\max}^{2}.$$ 
 
 On the other hand, notice that $A=\{p\in M;\,\,u(p)=u_{\max}\}\subset M_{\alpha}$ and hence, 
 
 $$\max_{A}\Psi \leq \max_{M_{\alpha}}\Psi\leq \max_{\partial M_{\alpha}} \Psi\leq \frac{(n\lambda-R)}{n}u_{\max}^{2}.$$ But, taking into account that $$\max_{A}\Psi =\frac{(n\lambda-R)}{n}u_{\max}^{2},$$ we conclude that $\Psi$ is constant and its value coincides with $\frac{(n\lambda-R)}{n}u_{\max}^{2}$ in $M_{\alpha}.$ Since $\alpha> 0$ can be chosen arbitrarily small, we then obtain that $\Psi= \frac{(n\lambda-R)}{n}u_{\max}^{2}$ on $M^n.$ In particular, returning to (\ref{eq76u}), we conclude that $M^n$ is Einstein. Thus, it suffices to use Proposition 2.4 of \cite{Petersen-Chenxu} to conclude that $M^n$ is homothetic to the standard hemisphere $\Bbb{S}^n_+.$

Finally, since (\ref{eqght}) implies that $M^n$ is homothetic to the standard hemisphere $\Bbb{S}^n_+,$ we therefore conclude that $$\max_{i}\kappa(\partial M_{i})\geq \sqrt{\frac{n\lambda-R}{mn}}$$ and this proves the requested result.

\end{proof}

\subsection{Proof of Corollary \ref{TcorB}}

\begin{proof}
Since $M^n$ has constant scalar curvature, we already know that $R< n\lambda.$ Next, upon integrating (\ref{laplaciano}) over $M^n,$ we use the Stokes' formula to obtain

\begin{eqnarray}
\label{eq67y}
-|\nabla u|_{|\partial M} |\partial M| &= &\frac{(R-n\lambda)}{m}\int_{M}u dM.
\end{eqnarray} Rearranging terms, one sees that

\begin{eqnarray}
\label{e4r5}
|\partial M| &\leq &\frac{(n\lambda-R)}{m}\frac{u_{\max}}{|\nabla u|_{|\partial M}} Vol(M),
\end{eqnarray} where $u_{\max}$ is the maximum value of $u$ on $M^n$ and $|\nabla u|_{| \partial M}$ denotes $|\nabla u|$ restricted to the boundary $\partial M.$

In order to proceed, we have to use the following inequality

\begin{equation}
\label{eqghtA}
\frac{|\nabla u|^{2}_{|\partial M}}{u_{\max}^{2}}\geq \frac{(n\lambda-R)}{mn},
\end{equation} which was established in Theorem \ref{thmK}.  Then, plugging (\ref{eqghtA}) into (\ref{e4r5}) we arrive at
\begin{eqnarray}
\label{eq658}
|\partial M|&\leq & \frac{n\lambda-R}{m}\sqrt{\frac{mn}{n\lambda-R}}\,Vol(M)\nonumber\\&=&\sqrt{\frac{n(n\lambda-R)}{m}}\,Vol(M).
\end{eqnarray}

On the other hand, it follows from (\ref{eq3e4}) that

\begin{eqnarray}
\label{lfmj}
n\lambda-R&\leq & n\lambda -\frac{n(n-1)}{m+n-1}\lambda\nonumber\\&=&\frac{mn}{m+n-1}\lambda,
\end{eqnarray} consequently,

\begin{equation*}
\sqrt{\frac{n(n\lambda-R)}{m}}\leq n\sqrt{\frac{\lambda}{m+n-1}},
\end{equation*} which substituted into (\ref{eq658}) yields

\begin{equation}
\label{eq6580}
|\partial M|\leq n\sqrt{\frac{\lambda}{m+n-1}} \,Vol(M).
\end{equation} This therefore gives the desired inequality. 

Finally, if equality holds in (\ref{eq6580}), then (\ref{lfmj}) also becomes an equality. Hence, we use again (\ref{eq3e4}) to infer $$R=\frac{n(n-1)}{m+n-1}\lambda.$$ Then, it suffices to apply \cite[Proposition 5.3]{He-Petersen-Wylie2012} to deduce that $M^n$ is Einstein and in this case, the result follows from Proposition 2.4 of \cite{Petersen-Chenxu}. The converse is obviously true. So, the proof is completed. 

\end{proof}

\section{The proof of Theorem \ref{thmBY}}
	\label{secProofBY}
	
	In this section, we are going to present the proof of Theorem \ref{thmBY}. In the first part of
the proof, we shall adapt some arguments outlined in \cite{Yuan,Qing,HMR,YuanW}. Initially, let $\big(M^{n},\,g,\,u,\,\lambda \big),$ $n\geq 3,$ be a nontrivial compact $m$-quasi-Einstein manifold.  Furthermore, in order to simplify the notation used in the computation that follows, we set a function $h:M\to \Bbb{R}$ given by 
	
	\begin{equation}
	\label{eqh}
	h:=\left(1+\alpha u\right)^{-\frac{n-2}{2}},
	\end{equation} where $$\alpha ^{-1}:=\max_{M}\left(u^{2}+\frac{m+n-1}{\lambda}|\nabla u|^{2}\right)^{\frac{1}{2}}.$$ Next, we consider the conformal metric $\overline{g}$ given by $$\overline{g}=h^{\frac{4}{n-2}}g.$$ In particular, it is well known that the scalar curvature of $M^n$ with respect to $\overline{g}$ satisfies 
	
	\begin{equation}
	\label{Rgb}
	R_{\overline{g}}=h^{-\frac{(n+2)}{(n-2)}}\left(R_{g}h-4\frac{(n-1)}{(n-2)}\Delta_{g}h\right).
	\end{equation} 
	
	In the sequel, we shall prove the following key lemma.

	\begin{lemma}\label{conformal  scalar curvature}
		Let $\big(M^{n},\,g,\,u,\,\lambda \big),$ $n\geq 3,$ be a nontrivial compact $m$-quasi-Einstein manifold. Then  the scalar curvature $R_{\overline{g}}$ with respect to the conformal metric $\overline{g}$ is nonnegative.  Moreover, $R_{\overline{g}}=0$ if and only if $R_{g}=\frac{n(n-1)}{m+n-1}\lambda$ and $u^{2}+\frac {m+n-1}{\lambda}|\nabla u|^{2}$ is constant on $M.$
		 
	\end{lemma}
	
	\begin{proof} A straightforward computation yields
		
		\begin{eqnarray}\label{eq1 proof th1}
		\Delta_{g}h&=& \Delta_{g}\left(1+\alpha u\right)^{-\frac{n-2}{2}} \nonumber \\
		&=&\div \left( \nabla (1+\alpha u)^{-\frac{n-2}{2}}\right)\nonumber \\ 
		&=&\div \left(-\frac{n-2}{2}(1+\alpha u)^{-\frac{n}{2}}\alpha \nabla u\right) \nonumber \\ 
		&=&-\left(\frac{n-2}{2}\right)\alpha(1+\alpha u)^{-\frac{n}{2}} \Delta u -\left(\frac{n-2}{2}\right)\left(-\frac{n}{2}\right)(1+\alpha u)^{-\frac{n}{2}-1}\alpha ^{2}|\nabla u|^{2} \nonumber \\
		&=&-\left(\frac{n-2}{2}\right)\alpha(1+\alpha u)^{-\frac{n}{2}}\Delta u+\frac{n(n-2)}{4}\alpha^{2}(1+\alpha u)^{-\frac{n+2}{2}}|\nabla u|^{2}. 
		\end{eqnarray}

On the other hand, it follows from (\ref{eq3e4}) and (\ref{laplaciano}) that 
		
		\begin{eqnarray*}
			\Delta u &= &\frac{u}{m}(R-\lambda n) \\
			&\geq& \frac{u}{m} \left(\frac{n(n-1)\lambda}{m+n-1}-\lambda n\right)\\
			&=& \frac{u}{m}n \lambda \bigg(\frac{n-1-m-n+1}{m+n-1}\bigg) \\
			&=&-\frac{un\lambda }{m+n-1}.
		\end{eqnarray*} Plugging this into (\ref{eq1 proof th1}), one sees that 
		
		\begin{eqnarray*}
			\Delta_{g}h \leq  \left(\frac{n-2}{2}\right)\alpha (1+\alpha u)^{-\frac{n}{2}}\frac{un\lambda }{m+n-1}+\frac{n(n-2)}{4}\alpha ^{2}(1+\alpha u)^{-\frac{n+2}{2}}|\nabla u|^{2}, 
		\end{eqnarray*} and rearranging the terms, we obtain

		\begin{eqnarray}\label{eq2 proof th1}
		\Delta_{g}h&\leq \frac{n(n-2)}{4}\alpha (1+\alpha u)^{-\frac{n+2}{2}}\left[\frac{2\lambda}{(m+n-1)}(1+\alpha u)u+\alpha |\nabla u|^{2}\right].
		\end{eqnarray}

		At the same time, it follows from (\ref{Rgb}) that 
		
		\begin{eqnarray*}
			\Delta_{g}h &=& \frac{(n-2)}{4(n-1)}R_{g}h-\frac{(n-2)}{4(n-1)}h^{\frac{n+2}{n-2}}	R_{\overline{g}} \\
			&=&\frac{(n-2)}{4(n-1)}R_{g}(1+\alpha u)^{-\frac{n-2}{2}}-\frac{(n-2)}{4(n-1)}\left((1+\alpha u)^{-\frac{(n-2)}{2}}\right)^{\frac{(n+2)}{(n-2)}}R_{\overline{g}}\\
			&=&\frac{(n-2)}{4(n-1)}R_{g}(1+\alpha u)^{-\frac{n-2}{2}}-\frac{(n-2)}{4(n-1)}(1+\alpha u)^{-\frac{n+2}{2}}R_{\overline{g}}.
		\end{eqnarray*}
		
		Now, substituting this data into (\ref{eq2 proof th1}) and simplifying the terms, we arrive at

		\begin{eqnarray*}
			R_{g}(1+\alpha u)^{2}-R_{\overline{g}} \leq n(n-1)\alpha \left[\frac{2\lambda}{m+n-1}(1+\alpha u)u+\alpha |\nabla u |^{2}\right], 
		\end{eqnarray*} which also can be rewritten as 
		
		\begin{eqnarray}
		R_{\overline{g}} \geq R_{g}(1+\alpha u)^{2}-n(n-1)\alpha \left[\frac{2\lambda}{m+n-1}(1+\alpha u)u+\alpha |\nabla u |^{2}\right].
		\end{eqnarray}

		Next, using again (\ref{eq3e4}), we can write the previous expression as 
		
		\begin{eqnarray*}
			R_{\overline{g}}&\geq & \frac{n(n-1)\lambda}{m+n-1}(1+\alpha u)^{2}-n(n-1)\alpha\left[\frac{2\lambda}{m+n-1}(1+\alpha u)u+\alpha |\nabla u |^{2}\right] \\
			&=&\frac{n(n-1)}{m+n-1}\lambda\left[ (1+\alpha u)^{2}-2\alpha (1+\alpha u)u-\left(\frac{m+n-1}{\lambda}\right)\alpha^{2}|\nabla u|^{2}\right],
		\end{eqnarray*} so that 
		
		\begin{eqnarray*}
			R_{\overline{g}}&\geq & \frac{n(n-1)}{m+n-1}\lambda\left[ 1-\alpha ^{2}u^{2}-\left(\frac{m+n-1}{\lambda}\right)\alpha^{2}|\nabla u|^{2}\right].\nonumber\\
		\end{eqnarray*} Finally, we obtain

		\begin{eqnarray}\label{eq3 proof th1}
		R_{\overline{g}}\geq \frac{n(n-1)}{m+n-1}\lambda \left[1-\alpha^{2}\left(u^{2}+\frac{m+n-1}{\lambda}|\nabla u|^{2}\right)\right]
		\end{eqnarray} and therefore, our choice of $\alpha$ guarantees that $R_{\overline{g}}\geq 0.$ Moreover,  $R_{\overline{g}} = 0$ if and only if (\ref{eq3e4}) becomes an equality. Hence, also from (\ref{eq3 proof th1}), $R_{\overline{g}} = 0$ if and only if $R_{g}=\frac{n(n-1)}{m+n-1}\lambda$ and $u^{2}+\frac {m+n-1}{\lambda}|\nabla u|^{2}$ is constant on $M.$ This finishes the proof of the lemma. 
		
	\end{proof}

	\subsection{Conclusion of the proof of Theorem \ref{thmBY}}
	\begin{proof} To begin with, we need to prove that the mean curvature $H_{\overline{g}}^{i}$ of $\partial M_{i}$ with respect to $\overline{g}$ is strictly positive. Indeed, since $u$ vanishes on the boundary, it follows from (\ref{eqh}) that $h=1$ on $\partial M_{i}.$ Therefore, $\overline{g}=g$ on $\partial M_{i}$ and $(\partial M_{i},\,\overline{g})$ is isometric to $(\partial M_{i},\,g),$ which can be also isometrically embedded into $\Bbb{R}^n$ as a convex hypersurface with the same mean curvature $H_{0}^{i}$ (induced by the Euclidean metric). Now, we recall that the mean curvature of $\partial M_{i}$ with respect to $\overline{g}$ is given by
		
		\begin{equation}
		H_{\overline{g}}^{i}=h^{-\frac{2}{n-2}}\left(H_{g}^{i}+2\frac{(n-1)}{(n-2)}\partial_{N}\log h\right).
		\end{equation} In particular, by recalling that $H_{g}^{i}=0,$ one sees that
		
		\begin{eqnarray}
		\label{eqlkp}
		H_{\overline{g}}^{i}&=&2\frac{(n-1)}{(n-2)}\langle \nabla (1+\alpha u)^{-\frac{n-2}{2}},\,N\rangle\nonumber\\&=& (n-1)\alpha|\nabla u|_{|\partial M_{i}}
		\end{eqnarray} and therefore, $H_{\overline{g}}^{i}>0,$ as we wanted to prove. 
		
		Proceeding, by (\ref{eqlkp}) we obtain
		
		\begin{eqnarray}
		\mathfrak{m}_{BY}(\partial M_{i},\,\overline{g})&=&\int_{\partial M_{i}}\left(H_{0}^{i}-H_{\overline{g}}^{i}\right)dS_{\overline{g}}\nonumber\\
		&=&\mathfrak{m}_{BY}(\partial M_{i},\,g)-(n-1)\alpha |\nabla u|_{|\partial M_{i}} |\partial M_{i}|.
		\end{eqnarray} Thus, we invoke Lemma \ref{conformal  scalar curvature} to deduce that $R_{\overline{g}}\geq 0,$ and taking into account that $H_{\overline{g}}^{i}>0,$ we are in a position to apply the positive mass theorem for Brown-York mass \cite{ShiTam,SY3} to conclude that $\mathfrak{m}_{BY}(\partial M_{i},\,\overline{g})\geq 0.$ Hence, we obtain 
		\begin{equation}
		\label{lkj1}
		|\partial M_{i}|\leq \frac{1}{(n-1)\alpha |\nabla u|_{| \partial M_{i}}}\,\mathfrak{m}_{BY}(\partial M_{i},\,g)=\frac{1}{(n-1)\alpha |\nabla u|_{| \partial M_{i}}}\,\int_{\partial M_{i}}H_{0}^{i}dS_{g}.
		\end{equation} This proves the desired estimate. 
		
		Finally, if equality holds in (\ref{lkj1}) for some component $\partial M _{i_{0}}$, then we necessarily have $$\mathfrak{m}_{BY}(\partial M_{i_{0}},\,\overline{g})=0.$$ In this situation, we may invoke the equality case of the positive mass theorem for Brown-York mass to conclude that the conformal metric $\overline{g}$ is flat and $(M, \overline{g})$ is isometric to a bounded domain in $\mathbb{R}^{n}.$ Therefore, it follows by Lemma $\ref{conformal  scalar curvature}$ that $R_{g}=\frac{n(n-1)\lambda}{m+n-1}.$ Now we can use \cite[Proposition 5.3]{He-Petersen-Wylie2012} to deduce that $M^{n}$ is an Einstein manifold and hence, it follows from \cite[Proposition 2.4]{Petersen-Chenxu} that $(M^{n},\,g)$ is homothetic to the standard hemisphere $\mathbb{S}^{n}_{+}$ given by Example \ref{example1}. Conversely, if $(M^{n},g)$ is homothetic to $\mathbb{S}^{n}_{+},$ then $\overline{g}$ is a flat metric and the Brown-York mass $\mathfrak{m}_{BY}(\partial M,\,\overline{g})$ vanishes. Therefore, equality holds in (\ref{lkj1}). So, the proof is completed.

	\end{proof}

\begin{acknowledgement} The authors want to thank the referee for the careful reading, relevant remarks and valuable suggestions. The third named author would like to thank the Instituto de Mate\-m\'atica - Universidade Federal Fluminense, where part of this work was carried out, for the fruitful research environment. He is grateful to Detang Zhou for the warm hospitality and enlightening conversations on the related topics.
\end{acknowledgement}

\end{document}